\documentclass[article,11pt,reqno]{amsart}
\usepackage[english]{babel}
\usepackage[utf8]{inputenc}
\usepackage{thm-restate}
\usepackage{amsmath, amssymb, amsthm}
\usepackage[table]{xcolor}
\usepackage{tikz-cd}
\usepackage{graphicx}
\usepackage[]{todonotes}
\usepackage{physics}
\usepackage{setspace}
\usepackage[top=0.85in,bottom=0.85in,left=0.95in,right=0.95in]{geometry}

\usepackage[hidelinks]{hyperref}

\usepackage{caption}
\usepackage{subcaption}
\usepackage{floatrow}
\usepackage{xpatch}

\usepackage{graphicx}
\usepackage{mathrsfs}
\usepackage{bbm}

\usepackage{verbatim}
\usepackage{adjustbox}
\newcommand{\nocontentsline}[3]{}
\newcommand{\tocless}[2]{\bgroup\let\addcontentsline=\nocontentsline#1{#2}\egroup}

\DeclareMathOperator*{\UEG}{UEG}

\newcommand{\N}{\mathbb{N}}

\newcommand{\Z}{\mathbb{Z}}

\newcommand{\cc}{\leftrightarrow}

\newcommand{\Prbcur}{\mathbf{P}}

\usepackage{tikz}
\usepackage{pgfplots}
\usepackage{tikz-3dplot}
\usetikzlibrary{positioning}
\usetikzlibrary{arrows}

\newcommand{\id}{1\! \!1}

\newcommand{\nn}{\mathbf{n}}

\newtheorem{theorem}{Theorem}[section]
\newtheorem{definition}[theorem]{Definition}

\newtheorem{remark}[theorem]{Remark}
\newtheorem{corollary}[theorem]{Corollary}

\newtheorem{question}[theorem]{Question}

\newtheorem{lemma}[theorem]{Lemma}

\newcommand{\dc}{\Prbcur^{\emptyset, \emptyset}}
\usepackage{cleveref}
\usetikzlibrary{positioning,calc}

\newlength\tindent
\makeatletter
\newcommand{\changeoperator}[1]{%
  \csletcs{#1@saved}{#1@}%
  \csdef{#1@}{\changed@operator{#1}}%
}
\newcommand{\changed@operator}[1]{%
  \mathop{%
    \mathchoice{\textstyle\csuse{#1@saved}}
               {\csuse{#1@saved}}
               {\csuse{#1@saved}}
               {\csuse{#1@saved}}%
  }%
}

\usepackage{etoolbox}
\makeatletter
\patchcmd{\@maketitle}{\topskip42\p@}{\topskip10\p@}{}{}
\makeatother

\usetikzlibrary{calc,arrows.meta}

\author{Ulrik Thinggaard Hansen}
\address{Ulrik Thinggaard Hansen \\ Department of Mathematics,
Universität Innsbruck, Technikerstrasse 13, 6020 Innsbruck, Austria }
\email{ulrik.hansen@uibk.ac.at}

\author{Frederik Ravn Klausen}
\address{Frederik Ravn Klausen, University of Cambridge, DPMMS, Cambridge, United Kingdom}
\email{frk23@cam.ac.uk}

\usepackage{sidecap}
\usetikzlibrary{patterns,decorations.pathmorphing,arrows.meta,calc}
\usepackage{wrapfig}
\changeoperator{prod}
\pgfplotsset{compat=1.18}

\usepackage{mdframed}
\mdfdefinestyle{MyFrame}{
    linecolor=blue,
    outerlinewidth=30pt,
    roundcorner=20pt,
    innertopmargin=4pt,
    innerbottommargin=5pt,
    innerrightmargin=20pt,
    innerleftmargin=20pt,
    backgroundcolor=gray!15!white}

\title{\vspace{-5cm}The End of the Double Random Current}

\begin{document}

\maketitle

\begin{abstract}
The Ising correlation function can be expressed in terms of connectivity probabilities of a random graph known as the double random current. In 2016, Duminil-Copin asked whether this random graph only has one topological end.
One-endedness implies that free and wired random-cluster measures are equal and that any translation-invariant Gibbs measure of the Ising model is a convex combination of the $+$ and $-$ Gibbs measures. In the general setting of transitive amenable graphs, these properties have been established by Raoufi, effectively circumventing the question of one-endedness.
In this paper, we prove one-endedness of the double random current on any transitive, amenable one-ended graph and,
relying on Raoufi's work, we answer the question of Duminil-Copin in the affirmative in this general setup. The results of this paper do not give a new proof of uniqueness of random-cluster measure or the characterisations of Gibbs states for the Ising model.
\end{abstract}

\section{Introduction}
The zoo of graphical representations of the Ising model  is vast, including models such as the random-cluster, random current and loop O(1) models, as well as the recent addition of the dual of the random current \cite{lopez2026ising}. These are random graphs which represent correlations in terms of connection probabilities, and more delicately, understanding their coarse stochastic geometry unlocks many qualitative properties of the Ising model. Famously, long-range order on the spin side of the Ising model corresponds to existence of infinite clusters of either the FK or double current representation on the geometric side.
Percolation of the supercritical single random current  or loop O(1) model implies exponential decay of truncated correlations.
Finally, one-endedness of a certain double random current implies a characterization of the translation-invariant Gibbs measures of the Ising model \cite{raoufi2020translation}. For percolation, the similar study of coarse geometry on general transitive graphs has been very fruitful, starting with the seminal paper by Benjamini and Schramm \cite{BS}.

For any locally finite graph $G=(V,E)$, let $\Omega_\emptyset(G)$ be the set of even subgraphs of $G$, that is the set of graphs $(V,F)$, with $F \subseteq E$, such that every vertex in $V$ is incident to an even number of edges in $F$.
On a finite graph $G$, the loop O(1) model with parameter $x \in (0,1)$,  $\ell_{G,x}$, is the percolation measure which assigns every even subgraph $\eta$ a probability proportional to $x^{\abs{\eta}}$, where $\abs{\eta}$ denotes the number of edges in $\eta$. Equivalently, it is Bernoulli percolation at edge parameter $\frac{x}{1+x}$ conditioned to be even.

For any amenable, transitive graph $\mathbb{G}$ and any sequence of finite subgraphs $G_n \nearrow \mathbb{G}$, the local limit of $\ell_{G_n,x}$ exists (see, e.g., \cite[Appendix A]{hansen2026supercriticalloopo1random}) and we denote that  measure by $\ell_{\mathbb{G},x}$.

For percolation measures $\nu_1$ and $\nu_2$, if   $(\eta^1, \eta^2) \sim \nu_1 \otimes \nu_2$, we denote by $\nu_1 \cup \nu_2$ the law of $\eta^1 \cup \eta^2$. That is, $\nu_1\cup \nu_2$ is the law of the max of two independent samples. The double random current is the percolation measure $\Prbcur_{G,x}^{\emptyset,\emptyset} = \ell_{G,x} \cup \ell_{G,x} \cup \mathbb{P}_{G,x^2},$ where  $\mathbb{P}_{G,x^2}$ is Bernoulli percolation with parameter $x^2$. The FK-Ising a.k.a. random-cluster model is
$\phi_{G,x} = \ell_{G,x} \cup \mathbb{P}_{G,x}$.

It is known that these measures have infinite volume limits $\Prbcur_{\mathbb{G},x}^{\emptyset,\emptyset}$, and $\phi_{\mathbb{G},x}^0$ and when $\mathbb{G}$ is transitive and amenable, the infinite cluster is almost surely unique if it exists \cite{aizenman2015random, burton1991topological}.
Whenever $G_n \nearrow \mathbb{G}$, the number of infinite components of $\mathbb{G} \setminus G_n$ is an increasing sequence, so the limit exists and it is independent of the choice of exhaustion. If the limit is $k$, say that $\mathbb{G}$ has $k$ \emph{ends}. By deletion tolerance  $\phi_{\mathbb{G},x}^0$ is almost surely one-ended. The goal of this paper is to show the same for the double random current. 
\begin{theorem}\label{thm:double_current_has_one_end}
Let $\mathbb{G}$ be a transitive, one-ended and amenable graph. Whenever the double random current $\Prbcur^{\emptyset,\emptyset}_{\mathbb{G},x}$  percolates, almost surely it has a unique, one-ended infinite cluster.
\end{theorem}

For any finite subgraph $G \subset \mathbb{G}$, define the wired graph $G^1$ obtained from $\mathbb{G}/(G^c)$ by removing the internal edges of $G^c$ (so that $G^1$ is finite).
On a finite graph $G$, the (traced) single random current  can be defined through $\Prbcur_{G,x}^\emptyset = \ell_{G,x} \cup \mathbb{P}_{G,1 - \sqrt{1-x^2}}$; we usually drop the subscript $x$ in the following.
If $G_n \nearrow \mathbb{G}$, both  $\ell_{G_n,x}$ and $\ell_{G_n^1,x}$ have limits $\ell^0_{\mathbb{G}}$ and $\ell^+_{\mathbb{G}}$ \cite{aizenman2015random}. 
Similarly, $\Prbcur_{G_n}^\emptyset$ and $\Prbcur_{G^1_n}^\emptyset$ have limits  which satisfy that $\Prbcur_{\mathbb{G}}^\emptyset = \ell_{\mathbb{G}}^0 \cup \mathbb{P}_{\mathbb{G},1 - \sqrt{1-x^2}}$ and $\Prbcur_{\mathbb{G}}^+ = \ell_{\mathbb{G}}^+ \cup \mathbb{P}_{\mathbb{G},1 - \sqrt{1-x^2}}$, since the union $\cup$ is continuous. A priori, these limits are different, but in the setting of transitive, amenable graphs, Raoufi's uniqueness result $\phi_\mathbb{G}^1 = \phi_\mathbb{G}^0$  implies  $\Prbcur_\mathbb{G}^+ =\Prbcur_\mathbb{G}^\emptyset$ (see e.g., \cite[Appendix A]{hansen2026supercriticalloopo1random} for two proofs).

Define also the mixed double random current $\Prbcur_{\mathbb{G}}^{\emptyset,+} = \Prbcur_{\mathbb{G}}^\emptyset \cup \Prbcur_{\mathbb{G}}^+$, which equals $\Prbcur_{\mathbb{G}}^{\emptyset,\emptyset}$. In particular, \Cref{thm:double_current_has_one_end}  combined with the work of Raoufi \cite{raoufi2020translation}  answers \cite[Question 8]{DC16} in the affirmative. 
\begin{corollary}\label{cor:question}
    Let $\mathbb{G}$ be a transitive, one-ended and amenable graph. Whenever the double random current $\Prbcur^{\emptyset,+}_{\mathbb{G}}$  percolates, almost surely, it has a unique, one-ended infinite cluster.
\end{corollary}

\subsection{Prior work. Coarse geometry of the double current and the Ising model}

In \cite{raoufi2020translation}, Raoufi explained, using the switching lemma for the double random current, how the coarse geometry of the mixed current $\Prbcur^{\emptyset,+}$ yielded conclusions for the random-cluster measure and Ising model.
That is also the reason that Duminil-Copin's question  was stated in terms of the double random current $\Prbcur_{\mathbb{G}}^{\emptyset,+}$.
Only the combination of Raoufi's result $\phi_\mathbb{G}^1 = \phi_\mathbb{G}^0$ with \Cref{thm:double_current_has_one_end} shows that all classically studied versions of the double random current are one-ended.

Raoufi gave a conceptual coarse geometric proof of the characterisation of translation invariant Gibbs states of the Ising model as convex combinations of the $+$ and $-$ states. In general, this result is obvious for $\beta<\beta_c,$ and the full result was already known in the nearest-neighbour setup on $\Z^d$:
For $d=2,$ it was first reached by Aizenman \cite{aizenman1980translation} and Higuchi \cite{higuchi1981absence} (in fact all Gibbs states and not just the translation invariant ones). For general $d\geq 3,$ this is due to Bodineau for $\beta > \beta_c$ \cite{bodineau2006translation}.  In the critical case $\beta = \beta_c,$ it follows from the work of Aizenman and Fernández \cite{aizenman1986critical} for $d \geq 4$ and Aizenman, Duminil-Copin and Sidoravicius for $d=3$ \cite{aizenman2015random}.

\subsection{Proof idea.} The idea of the proof is to argue via contradiction by focusing on the even subgraphs of $\mathbf{P}^{\emptyset,\emptyset}_{\mathbb{G}}$. If $\mathbf{P}^{\emptyset,\emptyset}_{\mathbb{G}}$ is two-ended, these come in two distinct classes, corresponding to whether the total parity of the number of edge-disjoint infinite paths is even or odd. We will refer to these classes as having even, respectively odd, flow.

The coupling $\ell_{G,x}\otimes \ell_{G,x}\otimes \mathbb{P}_{G,x^2}$ gives us two connections between $\mathbf{P}^{\emptyset,\emptyset}$ and $\ell$: First, that $\mathbf{P}^{\emptyset,\emptyset}_{G}=\ell_{G,x}\cup \ell_{G,x}\cup \mathbb{P}_{G,x^2}$ and second, that $\ell_{G,x}$ is the uniform even subgraph of $\mathbf{P}^{\emptyset,\emptyset}_{G}$. We are going to argue that the first identity together with uniqueness of the infinite component implies that, under $\ell_{G,x}\otimes \ell_{G,x}\otimes \mathbb{P}_{G,x^2},$ a sample $\eta^1$ (corresponding to the first $\ell_{G,x}$)  with high probability looks like it is an even subgraph of the double current $\nn$ with odd flow. However, by studying uniform even subgraphs of finite graphs, the second identity implies that, again under $\ell_{G,x}\otimes \ell_{G,x}\otimes \mathbb{P}_{G,x^2},$ $\eta^1$ looks like an even subgraph of $\nn$ with even flow with probability at least one half (corresponding to the case where the finite volume approximations of the two ends of $\nn$ crossing a box are connected outside of a large box). This yields the contradiction.

\section{Proof of the main result}

\subsection{Preliminaries}
For a graph $G$, finite or infinite, let $\Omega_\emptyset(G)$ denote the even subgraphs of $G$.

For finite graphs $G$, by \cite[Proposition 2.2]{hansen2025general}, we will consider the double random current triplet $(\eta^1,\eta^2,\delta)\sim \ell_{G,x}\otimes \ell_{G,x}\otimes \mathbb{P}_{G,x^2},$ which satisfies that 
$$
\nn:=\eta^1\cup\eta^2\cup\delta\sim \mathbf{P}^{\emptyset,\emptyset}_{G,x} \qquad \mathrm{and}\qquad  \ell_{G,x}\otimes \ell_{G,x}\otimes \mathbb{P}_{G,x^2}[\eta^1 \in \cdot\mid \nn]={\UEG}_{\nn}[\;\cdot\;],
$$
where $\UEG_{\nn}$ denotes the uniform measure\footnote{A weaker version of this result appeared previously in \cite{klausen2021monotonicity} (and \cite{lis2017planar} implicitly). Reasoning with this extended coupling was also pivotal in the
 proof of \cite[Proposition 2.12]{lopez2026ising}.} on $\Omega_{\emptyset}(\nn)$.
From now on, we drop the parameter $x$ from the notation.

We let $\Lambda_N$ denote the box of size $N$ for the graph distance and define the annulus $A_{n,N} = \Lambda_N \setminus \Lambda_n$.

Let $2\mathtt{Ends}_n$ denote the event that $\nn \setminus \Lambda_n$ has two infinite clusters, $2\mathtt{Cross}_{n,N}$ be the event that the annulus $A_{n,N}$ has exactly two crossing clusters, and let $2\mathtt{EndsCross}_{n,N}=2\mathtt{Ends}_n \cap 2\mathtt{Cross}_{n,N}.$ Define also $2\mathtt{EndsCross}_{r,n,N} = 2\mathtt{EndsCross}_{r,n} \cap 2\mathtt{EndsCross}_{n,N}$. 

For a graph $G=(V,E)$,  any $F \subseteq E$, and any configuration $\eta \in \{0,1\}^E$, we use the notation $\eta_F$ for the restriction of $\eta$ to $F$, defined by $\eta_F(e) = \eta(e)\id[e\in F]$, for every $e\in E$.

For  any infinite, two-ended, connected graph $\mathbb{W}$, and any finite graph $F$ such that $\mathbb{W} \setminus F$ has two infinite components $\mathcal{C}_1(F),\mathcal{C}_2(F)$, define $L(\eta)$ by
$$
L(\eta) = \abs{\partial(\eta_F) \cap \mathcal{C}_1(F)} =\abs{\partial(\eta_F) \cap \mathcal{C}_2(F)} \pmod{2}.
$$
One may check that $L(\eta)$ is independent of $F$. Note, furthermore, that $L(\eta)=1$ implies that $\eta$ percolates (i.e. has an infinite connected component).  Define the odd-flow even subgraphs of $\mathbb{W}$,
\begin{align*}
    \Omega_\emptyset^{\infty,\infty}(\mathbb{W}) =  \{ \eta \in \Omega_\emptyset(\mathbb{W}) \mid  L(\eta) = 1 \}.
\end{align*}
If $\mathbb{W}$ is two-ended, $\Omega_{\emptyset}^{\infty,\infty}(\mathbb{W})$ and the free even subgraphs $\Omega_{\emptyset}^{0}(\mathbb{W})=\{L(\eta)=0\}$ form a partition of $\Omega_{\emptyset}(\mathbb{W})$ and are the supports of the two extremal Gibbs measures for $\UEG_{\mathbb{W}}.$ This follows from the general characterisation of Gibbs states for the uniform even subgraph
\cite[Theorem 3.14, Proposition 3.16]{hansen2023uniform}.

Let us start with two general lemmas. The following lemma is stated for the free measure $\ell^0_{\mathbb{G}}$, but replacing \cite[Proposition A.2]{hansen2026supercriticalloopo1random} by  \cite[Proposition A.4]{hansen2026supercriticalloopo1random}, the same proof works for $\ell^+_{\mathbb{G}}$.

\begin{lemma}\label{lemma:ergodicity}
Let $\mathbb{G}$ be a vertex-transitive, amenable, infinite graph. Then, the double random current  triplet
$\ell^0_{\mathbb{G}}\otimes \ell^0_{\mathbb{G}}\otimes \mathbb{P}_{\mathbb{G}}$ is tail-trivial and ergodic.
Moreover, $\ell^0_{\mathbb{G}}\cup \ell^0_{\mathbb{G}}\cup \mathbb{P}_{\mathbb{G}}$ has at most one infinite cluster almost surely.
\end{lemma}
\begin{proof}
First, we argue that $\ell^0_{\mathbb{G}}$ is tail-trivial and ergodic. Let $A\subseteq \{0,1\}^{E(\mathbb{G})}$ be a tail event. Then, by the infinite volume loop-cluster coupling\footnote{Since $\omega$ is almost-surely one-ended there is only one uniform even subgraph \cite[Theorem. 3.14]{hansen2023uniform}, denoted $\operatorname{UEG}_{\omega}$.}  \cite[Proposition A.2]{hansen2026supercriticalloopo1random}, 
$$
\ell^0_{\mathbb{G}}[A]=\phi^0_{\mathbb{G}}[\operatorname{UEG}_{\omega}[A]]=\phi^0_{\mathbb{G}}[\operatorname{UEG}_{\omega}[A]\cap \{\#\mathtt{Ends}(\omega)\leq 1\}]=\phi^0_{\mathbb{G}}[\{\operatorname{UEG}_{\omega}[A]=1\}\cap \{\# \mathtt{Ends}(\omega)\leq 1\}],
$$
since, by amenability, $\phi^0_{\mathbb{G}}$ almost surely has at most one end and $\operatorname{UEG}_{\omega}$ is tail-trivial for every one-ended\footnote{It is the unique Gibbs measure of the uniform even subgraph  \cite[Theorem 3.14]{hansen2023uniform}, and hence, in particular, extremal, which implies tail-triviality \cite[Chapter 6]{friedli2018}.} $\omega$.

Let us argue that $\{\operatorname{UEG}_{\omega}[A]=1\}\cap \{\# \mathtt{Ends}(\omega)\leq 1\}$ is a tail event. Indeed, fix any two configurations $\omega,\omega'$ that disagree on only finitely many edges and have at most one end. Let $n<N$ be such that $\omega$ and $\omega'$ agree on $\mathbb{G}\setminus \Lambda_n$ and such that there is at most one crossing from $\Lambda_n$ to $\partial_v\Lambda_N$ in $\omega$ (cf. the proof of \Cref{lemma:two_crossings}).  Then, if $(\eta_n,\eta_N)\sim\operatorname{UEG}_{\omega_{\mathbb{G}\setminus \Lambda_n}}\otimes \operatorname{UEG}_{\omega_{\Lambda_{N}}},$ $\eta_n\Delta \eta_N\sim \operatorname{UEG}_{\omega}$ by \cite[Corollary 3.7]{hansen2023uniform}. In particular, $\operatorname{UEG}_{\omega}$ and $\operatorname{UEG}_{\omega'}$ can be coupled to agree outside of $\Lambda_N$. Since $A$ is a tail event, $\operatorname{UEG}_{\omega}[A]=\operatorname{UEG}_{\omega'}[A].$ Thus, $\{\operatorname{UEG}_{\omega}[A]=1\}\cap \{\# \mathtt{Ends}(\omega)\leq 1\}$ is a tail event. 

By the tail-triviality of $\phi^0_{\mathbb{G}}$ (see e.g., \cite[p.1113]{lyons2000phase}), we deduce tail-triviality of $\ell^0_{\mathbb{G}}.$ In particular, $\ell^0_{\mathbb{G}}$ is a tail-trivial, automorphism-invariant Gibbs measure for the loop O($1$) model on $\mathbb{G}$. Since products of tail-trivial measures are tail-trivial\footnote{see e.g., \cite[Example 7.18]{GeorgiiBook}}, we get the same conclusion for $\ell^0_{\mathbb{G}}\otimes \ell^0_{\mathbb{G}}\otimes \mathbb{P}_{\mathbb{G}}.$ As this measure has a natural Gibbs property\footnote{For any $G\subseteq \mathbb{G},$ the configuration in $G$ is conditionally independent of the one in $\mathbb{G}\setminus G$ given the parities on the vertex boundary $\partial_v G$ of the restrictions of the loop O($1$) factors to $G$ (or, equivalently, their restrictions to $\mathbb{G}\setminus G)$.}, and tail-trivial Gibbs measures are mixing by the backwards martingale convergence theorem\footnote{See, for instance, the proof of \cite[Theorem 6.58]{friedli2018}.}, we deduce that  $\ell^0_{\mathbb{G}}\otimes \ell^0_{\mathbb{G}}\otimes \mathbb{P}_{\mathbb{G}}$ is automorphism-invariant and mixing and hence ergodic.

Uniqueness of the infinite cluster (and ergodicity) for the double random current $\ell^0_{\mathbb{G}}\cup \ell^0_{\mathbb{G}}\cup \mathbb{P}_{\mathbb{G}}$ was proven in \cite[Theorem 2.5]{aizenman2015random}.
\end{proof}

Tail-triviality and ergodicity both mean that the number of ends is almost surely constant. The next general lemma moves two-endedness to finite volume.
\begin{lemma}\label{lemma:two_crossings}
    Let $\mathbb{G}$ be an infinite graph. Suppose that $\nu$ is a percolation measure under which there is a unique infinite cluster with $2$ ends almost surely. Then, for any $\varepsilon>0$, there exist $n,N$ large enough, such that $\nu[2\mathtt{Cross}_{n,N}]\geq 1-\varepsilon.$

\end{lemma}
\begin{proof}
Fix $n$ and let $\mathtt{CC}(n,N)$ be the random variable that counts the number of crossing clusters in the annulus $A_{n,N}$. Note that $\mathtt{CC}(n,N)$ is decreasing in $N$. Now, up to the $\nu$-null set that there are two one-ended infinite components, the event that there are two ends is the event
$
\bigcup_{n,N}  \bigcap_{k\geq N} \{\mathtt{CC}(n,k)=2\}.
$
By continuity of measure,
$$\lim_{n\to \infty} \lim_{N\to\infty} \liminf_{k\geq N} \nu[\{\mathtt{CC}(n,k) =2 \}]
= \lim_{n\to\infty} \nu[2\mathtt{Ends}_n] = 1.$$

Accordingly, given $\varepsilon>0,$ there exist $n,N$ such that $\nu[\{\mathtt{CC}(n,k) =2 \}]\geq 1-\varepsilon$ for all $k\geq N.$ In particular, this proves the desired.
\end{proof}

\subsection{The main proof}

In the following argument, we will assume two-endedness, so a typical configuration will have exactly two crossings of a large annulus. We warm up with the following observation: For $\omega\in 2\mathtt{Cross}_{n,N},$ we denote by $\mathcal{C}_1(\omega)$ and $\mathcal{C}_2(\omega)$ an arbitrary labelling of the two clusters crossing $\Lambda_N\setminus \Lambda_n.$ If
$\eta \in \Omega_{\emptyset}(\omega)$, then
\begin{align} \label{eq:parity}
    \abs{\partial (\eta_{\Lambda_n})\cap \mathcal{C}_1}=    \abs{\partial (\eta_{\Lambda_n})\cap \mathcal{C}_2}
    =    \abs{\partial (\eta_{\Lambda_N})\cap \mathcal{C}_1}
    =    \abs{\partial (\eta_{\Lambda_N})\cap \mathcal{C}_2} \pmod 2.
\end{align}
If one (and hence all four) of these numbers is even/odd, we say that $\eta$ has even/odd flow (\textit{in} the annulus). 
\begin{definition}
    For $\omega\in 2\mathtt{Cross}_{n,N},$  let $\mathcal{C}_1$ be one of its two crossing clusters. Define 
    $$\mathtt{EvenFlow}_{n,N}(\omega) = \{ \eta \in \Omega_\emptyset(\omega) \mid \abs{\partial (\eta_{\Lambda_n})\cap \mathcal{C}_1} \equiv 0 \pmod{2}  \}$$
     $$\mathtt{OddFlow}_{n,N}(\omega) = \{ \eta \in \Omega_\emptyset(\omega) \mid \abs{\partial (\eta_{\Lambda_n})\cap \mathcal{C}_1} \equiv 1 \pmod{2}  \}.$$
\end{definition}
The following lemma shows that the probability of odd flow of a uniform even subgraph is at most 1/2 using a XOR trick.
\begin{lemma} \label{lemma:source_counting}
    Suppose that $\omega \in 2\mathtt{Cross}_{n,N}$ is a finite subgraph of $\mathbb{G}$.
    Then
    $$
    \operatorname{UEG}_{\omega}[\mathtt{OddFlow}_{n,N}(\omega)] \leq \frac{1}{2}.
    $$
 \end{lemma}

\begin{proof}
 Let $\mathcal{C}_1(\omega)$, $\mathcal{C}_2(\omega)$ be the two clusters crossing $A_{n,N}$. 
Suppose there exists $\gamma^0\in \mathtt{OddFlow}_{n,N}(\omega)$.  Then, the map $\gamma\mapsto \gamma\Delta \gamma^0$ is a bijection between odd and even flow
$$
\cdot \;\Delta \gamma^0: \mathtt{OddFlow}_{n,N}(\omega)\to \mathtt{EvenFlow}_{n,N}(\omega). 
$$
This proves $\operatorname{UEG}_{\omega}[ \mathtt{OddFlow}_{n,N}(\omega)] \leq \frac{1}{2}$, which is the desired. 
\end{proof}

Define the event that both even subgraphs have even flow in the triplet,  
$$\mathtt{DoubleEvenFlow}_{n,N} = \{(\eta^1, \eta^2, \delta)  \mid  \eta^1 \in  \mathtt{EvenFlow}_{n,N}(\eta^1 \cup \eta^2\cup \delta),  \eta^2 \in  \mathtt{EvenFlow}_{n,N}(\eta^1 \cup \eta^2\cup \delta)\}.$$

The next technical lemma will allow us to match sources disjointly whenever the parity above is even, see \Cref{fig:DW_event} for a graphical illustration.  This will allow us to show that it is impossible to have double even flow at a scale that witnesses the two ends. 

\begin{lemma}[Disjoint wirings of even crossings]\label{lemma:disjoint_wiring}
Let $r < n$, $\omega \in 2\mathtt{Cross}_{r,n}$ and  $\eta \in \mathtt{EvenFlow}_{r,n}(\omega)$.
Then there exists a $\gamma \subseteq \omega_{A_{r,n}}$ with $\partial \gamma = \partial(\eta_{\Lambda_n})$.
In particular,
$
\partial(\eta_{\Lambda_n}) \cap \mathcal{C}_1(\omega) \overset{\gamma}{\not \cc} \partial(\eta_{\Lambda_n}) \cap \mathcal{C}_2(\omega).
$
\end{lemma}
\begin{proof}
If a cluster $C$ of $\omega_{A_{r,n}}$ does not cross the annulus $A_{r,n}$ then
\begin{align}\label{eq:touches_even_number_of_times}
    \abs{\partial(\eta_{\Lambda_n}) \cap C} \in 2\N.
\end{align}
By the even flow assumption,  \eqref{eq:touches_even_number_of_times} also holds each of the two crossing clusters.
The vertices in each of these even sets can be paired, and for each pair $\mathfrak{p} = (vw)$, there is a  path $\gamma_\mathfrak{p}\subset \omega_{A_{r,n}}$ with endpoints $v$ and $w$, and hence, $\partial(\gamma_{\mathfrak{p}})=\{v,w\}$. Setting $\gamma = \triangle_{\mathfrak{p}} \gamma_\mathfrak{p}$ yields the desired subgraph.
\end{proof}

For a connected graph $G$, and a set of vertices $A$ of even cardinality, define also $\ell_{G}^A$ to be Bernoulli percolation with parameter $\frac{x}{1+x}$ conditioned on the event that the vertices with odd degree are exactly $A$.
The next lemma shows that as long as the flow is even, there is a positive probability of wiring up the sources disjointly (see \Cref{fig:DW_event}). 
\begin{lemma}\label{lemma:disjoint_wiring_new}
 Suppose $(\eta^1,\eta^2, \delta) \in \mathtt{DoubleEvenFlow}_{n,N}$, and $\omega = \eta^1 \cup \eta^2 \cup \delta \in 2\mathtt{EndsCross}_{r,n,N}$ . Denote the two infinite components of $\omega_{\Lambda_n^c}$ by $\mathfrak{e}$ and $\mathfrak{f}$. Then, for any $r<n<N,$
 \begin{align*}
     \ell^{\partial(\eta^1_{\Lambda_n^c})}_{\Lambda_n} \otimes \ell^{\partial(\eta^2_{\Lambda_n^c})}_{\Lambda_n}  \otimes \mathbb{P}_{\Lambda_n}[\mathfrak{e} \overset{\nn_{\Lambda_n}\hspace{-1pt}\cup \hspace{1pt}\omega_{\Lambda_n^c}}{ \cc} \mathfrak{f}] < 1.
 \end{align*}
\end{lemma}
\begin{proof}
A parity consideration using \eqref{eq:parity} shows that $(\eta^1,\eta^2, \delta) \in \mathtt{DoubleEvenFlow}_{r,n}$, so two applications of \Cref{lemma:disjoint_wiring} give $\gamma_1$ and $\gamma_2$ with $\gamma_i \subseteq \omega_{A_{r,n}}$ such that $\partial \gamma_i =\partial (\eta_{\Lambda_n}^i)$.
Since $\omega \in  2\mathtt{Ends}_{r}$, the ends $\mathfrak{e},\mathfrak{f}$ are separated in $\omega_{\Lambda_r^c}$ and therefore also in $\gamma_1\cup\gamma_2 \cup \omega_{\Lambda_n^c}$.
The configuration $(\gamma_1, \gamma_2,0)$ has positive probability.

\end{proof}

 The idea in the next important lemma is that disjoint wirings contradict uniqueness of the infinite cluster. From now on we ease notation by writing $\ell_\mathbb{G}$ for $\ell^0_\mathbb{G}$.

\begin{lemma}\label{lemma:even_flow_is_impossible} Suppose the double random current is almost surely two-ended.
For any positive integers $r <n<N$,  $$ \ell_{\mathbb{G}}  \otimes \ell_{\mathbb{G}} \otimes \mathbb{P}_{\mathbb{G}}[\mathtt{DoubleEvenFlow}_{n,N} , 2\mathtt{EndsCross}_{r,n,N}] = 0.$$
\end{lemma}
\begin{proof}

Let $\mathfrak{e}\overset{\nn}{\cc} \mathfrak{f}$ denote the event that $\nn$ has two ends which are connected to each other, i.e. for every box $\Lambda$ large enough to separate ends, the two infinite components of $\nn_{\Lambda^c}$ are connected inside $\nn_{\Lambda}$. Denote the $\sigma$-algebra generated by the restriction to $\Lambda_{n}^c$ by  $\mathcal{A}_{\Lambda_n^c}$.
By uniqueness of the infinite component, 
$$
1=\ell_{\mathbb{G}}\otimes \ell_{\mathbb{G}}\otimes \mathbb{P}_{\mathbb{G}}[\mathfrak{e}\overset{\nn}{\cc} \mathfrak{f}]=\ell_{\mathbb{G}}\otimes \ell_{\mathbb{G}}\otimes \mathbb{P}_{\mathbb{G}}[\ell_{\mathbb{G}}\otimes \ell_{\mathbb{G}}\otimes \mathbb{P}_{\mathbb{G}}[\mathfrak{e}\overset{\nn}{\cc} \mathfrak{f}\mid \mathcal{A}_{\Lambda_n^c}]].
$$
The Gibbs property yields $$
\ell_{\mathbb{G}}\otimes \ell_{\mathbb{G}}\otimes \mathbb{P}_{\mathbb{G}}[\mathfrak{e}\overset{\nn}{\cc} \mathfrak{f}\mid \mathcal{A}_{\Lambda_n^c}]=\ell_{\Lambda_n}^{\partial(\eta^1_{\Lambda_n^c})}\otimes \ell_{\Lambda_n}^{\partial(\eta^2_{\Lambda_n^c})}\otimes \mathbb{P}_{\Lambda_n}[\mathfrak{e}\overset{\nn_{\Lambda_n}\cup \nn_{\Lambda_n^c}}{\cc} \mathfrak{f}],
$$
and
we conclude that $\ell_{\Lambda_n}^{\partial(\eta^1_{\Lambda_n^c})}\otimes \ell_{\Lambda_n}^{\partial(\eta^2_{\Lambda_n^c})}\otimes \mathbb{P}_{\Lambda_n}[\mathfrak{e}\overset{\nn_{\Lambda_n}\cup \nn_{\Lambda_n^c}}{\cc} \mathfrak{f}]=1$ for almost every $(\eta^1_{\Lambda_n^c},\eta^2_{\Lambda_n^c},\delta_{\Lambda_n^c})$, since this variable is bounded above by 1.

By \Cref{lemma:disjoint_wiring_new}, $\mathtt{DoubleEvenFlow}_{n,N} \cap 2\mathtt{EndsCross}_{r,n,N}\subseteq \{\ell_{\Lambda_n}^{\partial(\eta^1_{\Lambda_n^c})}\otimes \ell_{\Lambda_n}^{\partial(\eta^2_{\Lambda_n^c})}\otimes \mathbb{P}_{\Lambda_n}[\mathfrak{e}\overset{\nn_{\Lambda_n}\cup \nn_{\Lambda_n^c}}{\cc} \mathfrak{f}]<1\},$ which proves the desired.
\end{proof}

 We are now ready to prove one-endedness of the double random current.

\begin{proof}[Proof of \Cref{thm:double_current_has_one_end}]
The proof in \cite{aizenman2015random} shows that the double random current has a unique infinite component on amenable transitive graphs (if it exists). By tail-triviality or ergodicity, the number of ends of the infinite cluster is almost surely constant. 
Furthermore, this infinite cluster has at most two ends almost surely (see, e.g., \cite[Exercise 7.24]{lyons2017probability}).
Assume for contradiction that it has exactly two ends almost surely.
Our goal is to arrive at a contradiction by proving that  there exist $\varepsilon<\frac{1}{2}$ and $n<N$ such that,
$$1-\varepsilon \leq \ell_{\mathbb{G}}\otimes \ell_{\mathbb{G}}\otimes \mathbb{P}_{\mathbb{G}}[\eta^1\in \mathtt{OddFlow}_{n,N}(\nn)] \leq \frac{1}{2}.$$ Each inequality is proven separately.

\begin{figure}
    \centering
\includegraphics[width=0.7\textwidth]{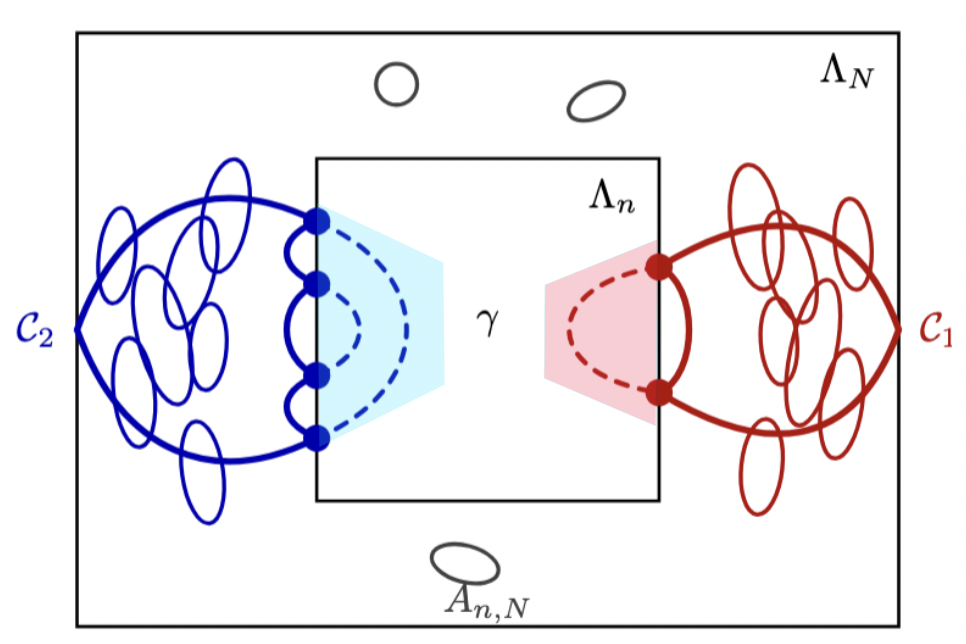}
\caption{The disjoint wiring of the sources of $\eta^1$. Two crossing clusters $\mathcal{C}_1$, $\mathcal{C}_2$ of $\nn$ in $A_{n,N}$ (solid) carry the sources $\partial(\eta^1_{\Lambda_n}) \cap \mathcal{C}_i$ on the inner boundary. The dashed even subgraph $\gamma \subset \Lambda_n$ satisfies $\partial\gamma = \partial(\eta^1_{\Lambda_n})$ and pairs sources within each $\mathcal{C}_i$, so that the crossing clusters are not $\gamma$-connected inside $\Lambda_n$. As explained in the proof of \Cref{lemma:disjoint_wiring_new},  the existence of these paths follows from the existence of two disjoint crossings of a smaller annulus (illustrated here with shadings).}    \label{fig:DW_event}
\end{figure}

\textbf{Lower bound:}
\Cref{lemma:two_crossings} and a union bound show that for every $\varepsilon>0$, we can choose $r$, $n$ and $N$ so large  that $\Prbcur_\mathbb{G}^{\emptyset,\emptyset}[2\mathtt{EndsCross}_{r,n,N}] \geq 1-\varepsilon$.
Using \Cref{lemma:even_flow_is_impossible} and symmetry of $\eta^1$ and $\eta^2$,  
\begin{align*}
    \frac{1}{2} &\leq \ell_{\mathbb{G}}  \otimes \ell_{\mathbb{G}} \otimes \mathbb{P}_{\mathbb{G}}\left[ \eta^1 \in \mathtt{OddFlow}_{n,N}(\nn)\mid 2\mathtt{EndsCross}_{r,n,N} \right] \\
&=\ell_{\mathbb{G}}  \otimes \ell_{\mathbb{G}} \otimes \mathbb{P}_{\mathbb{G}}[\eta^1 \in \Omega_\emptyset^{\infty,\infty}(\nn) \mid 2\mathtt{EndsCross}_{r,n,N}].
\end{align*}
Thus, $\ell_{\mathbb{G}}  \otimes \ell_{\mathbb{G}} \otimes \mathbb{P}_{\mathbb{G}}[\eta^1 \in \Omega_\emptyset^{\infty,\infty}(\nn)] > 0$  and the event is automorphism-invariant, so the probability is 1 by ergodicity (cf. \Cref{lemma:ergodicity}). Thus, by choosing $n,N$ large enough and applying \Cref{lemma:two_crossings},
\begin{align*}
    \ell_\mathbb{G}\otimes \ell_\mathbb{G}\otimes \mathbb{P}_\mathbb{G}[\eta^1 \in \mathtt{OddFlow}_{n,N}(\nn)] \geq \ell_\mathbb{G}\otimes \ell_\mathbb{G}\otimes \mathbb{P}_\mathbb{G}[\eta^1 \in \mathtt{OddFlow}_{n,N}(
    \nn)\mid 2\mathtt{EndsCross}_{r,n,N}] (1-\varepsilon) = 1- \varepsilon.
\end{align*}

\textbf{Upper bound:}
Let $G_k\nearrow \mathbb{G}$ and pick $n < N$ large as before and $M$ so large that $G_M \supset \Lambda_N$, then by \Cref{lemma:source_counting},
\begin{align*}
  \ell_{G_M}\otimes \ell_{G_M}\otimes \mathbb{P}_{G_M}[\eta^1 \in \mathtt{OddFlow}_{n,N}(\nn)]
  &\leq  \ell_{G_M}\otimes \ell_{G_M}\otimes \mathbb{P}_{G_M}[\eta^1 \in \mathtt{OddFlow}_{n,N}(\nn) \mid 2\mathtt{Cross}_{n,N}(\nn)] \\
    &=    \dc_{G_M}[\operatorname{UEG}_{\nn}[\mathtt{OddFlow}_{n,N}(\nn)]\mid 2\mathtt{Cross}_{n,N}(\nn)]
    \leq   \frac{1}{2}.
\end{align*}
As the event $\mathtt{OddFlow}_{n,N}$ is local (it depends only on the annulus $A_{n,N}$), weak convergence of the finite volume measures to their infinite volume counterparts yields
$\ell_{\mathbb{G}}  \otimes \ell_{\mathbb{G}} \otimes \mathbb{P}_{\mathbb{G}}[\mathtt{OddFlow}_{n,N}] \leq \frac{1}{2}.$
\end{proof}

\begin{remark}[One-endedness of $\Prbcur_{\mathbb{G}}^{+,+}$]
The same arguments show one-endedness of the plus-plus double random current $\Prbcur_{\mathbb{G}}^{+,+}$.
Assuming two-endedness, symmetry again yields 
    $
\ell^+_{\mathbb{G}}\otimes \ell^+_{\mathbb{G}}\otimes \mathbb{P}_{\mathbb{G}}[\eta^1 \in \Omega_\emptyset^{\infty,\infty}(\nn)] = 1.
$
Therefore, for $n,N$ large enough, 
     $$
\ell^+_{\mathbb{G}}\otimes \ell^+_{\mathbb{G}}\otimes \mathbb{P}_{\mathbb{G}}[\eta^1 \in \mathtt{OddFlow}_{n,N}(\nn)] \geq 1-\varepsilon.
$$
The upper bound goes through similarly, using that\footnote{A similar statement is also true in the mixed case, so that is not where the problem lies.} $\Prbcur_{G^1}^{\emptyset,\emptyset}[ \operatorname{UEG}_{\nn}[\eta]] = \ell_{G^1}[\eta]$. 

The symmetry of $\eta^1$ and $\eta^2$ seems essential, but must be overcome to prove one-endedness of the mixed current $\Prbcur_{\mathbb{G}}^{\emptyset,+}$. With Raoufi's result in hand, $\Prbcur_{\mathbb{G}}^+ = \Prbcur_{\mathbb{G}}^\emptyset$, which follows from $\phi_{\mathbb{G}}^0 = \phi_{\mathbb{G}}^1$ as explained before \Cref{cor:question}. This restores the symmetry, but it remains of interest to find an argument which circumvents it by other means.
\end{remark}

\begin{remark}
For readability, \Cref{thm:double_current_has_one_end} and \Cref{cor:question} are stated only for the free, nearest-neighbour model. 
The results hold just as well in the long-range setting. That is, for general weights $(x_{vw})_{v,w\in \mathbb{G}}$ satisfying:

\begin{enumerate}
    \item[] $\operatorname{Ferromagnetism}$: $x_{vw}\in [0,1]$ for all $v,w\in \mathbb{G}$
    \item[] $\operatorname{Automorphism \;invariance}$: $x_{\varphi(v)\varphi(w)}=x_{vw}$ for all $v,w\in \mathbb{G}$ and any $\varphi\in \mathrm{Aut}(\mathbb{G})$
    \item[] $\operatorname{One-endedness}$: The graph with edge set $\{vw\mid x_{vw}>0\}$ is one-ended \footnote{And hence, by automorphism invariance, connected.}.
    \item[] $\operatorname{Local\; finiteness}$: For fixed $v\in \mathbb{G},$ $\sum_{v,w}x_{v,w}<\infty.$
\end{enumerate}
Local finiteness ensures that with high probability for sufficiently large annuli $A_{n,N}$ there are no jumps from inside $\Lambda_n$ to outside $\Lambda_N^c$, which makes sure that the odd flow events are still well-defined. 
These weights must be used for both the loop O($1$) model and for the added Bernoulli percolation. 
\end{remark}

\section{Open problems}
Let us record some open problems regarding evens and ends. One-endedness of the random-cluster model follows directly from deletion tolerance and in \Cref{thm:double_current_has_one_end}, we proved one-endedness of the double random current; it is natural to ask about the remaining graphical representations. In contrast to the case for the double random current, we are not aware of any implications for the Ising model of these coarse geometrical properties.
 \begin{question}
     Whenever $\ell_{\Z^d}$ percolates, is the infinite cluster one-ended? What about the single random current?
 \end{question}

Our original approach to the question of one-endedness of the double random current went through the following reasoning that built some of the intuition of the proof, but a proof remains elusive. 
\begin{question}
    Suppose $\mathbb{G}$ is a transitive, amenable graph and $\mu$ and $\nu$ are percolation measures, where the laws of their uniform even subgraphs are equal. Does almost sure one-endedness of $\mu$ imply almost sure one-endedness of $\nu$?
\end{question}

In the ten years that have passed since Duminil-Copin's questions on random currents \cite{DC16}, the questions of conformal invariance of the critical single current on $\Z^2$ and exponential decay of truncated correlations on $\Z^d$ have been solved \cite{chen2023conformal,duminil2020exponential}, while there was only partial progress on some of the other questions \cite{aizenman2021marginal, aizenman2019emergent, duminil2021conformal2,hansen2023uniform}. One that remains is whether or not the random current on $\mathbb{Z}^d$ percolates throughout the supercritical regime. This would be resolved if the following holds:
\begin{question}
    Does the uniform even subgraph of any unimodular, transient graph $\mathbb{G}$ have an infinite component almost surely?
\end{question}

Let us end with an open question that builds on the perspective supported by this paper.
\begin{question}
Which additional qualitative properties of the phase diagram of the Ising model can be related to coarse stochastic geometrical properties of the graphical representations?
\end{question}

\section*{Acknowledgments}
We thank Lorca Heeney for insightful remarks. No data was used for this study and the authors have no relevant conflicts of interest.   FRK was supported by the Carlsberg Foundation, grant CF24-0466. This research was funded in part by the Austrian Science Fund (FWF) 10.55776/P34713.
 Towards the completion of this paper, we became aware of efforts by Gunaratnam, Panagiotis, Panis and Severo  \cite{GPPS}  to prove a stronger version of the main theorem of this paper on $\Z^d.$

\bibliographystyle{abbrv}
\bibliography{bibliography.bib}

\end{document}